\documentclass[reqno]{amsart}
\usepackage[foot]{amsaddr}

\usepackage{amsmath,amssymb,amsfonts,mathrsfs}

\usepackage{bbold}

\usepackage[english]{babel}
\usepackage[utf8]{inputenc}
\usepackage[T1]{fontenc}
\usepackage{mathtools}
\usepackage{lmodern}
\usepackage{comment}
\usepackage{cancel}

\usepackage[
top=2cm,
bottom=2cm,
left=3cm,
right=3cm,
marginparwidth=2.5cm
]{geometry}

\usepackage[colorinlistoftodos]{todonotes}
\usepackage[hypertexnames=false]{hyperref}
\hypersetup{
    colorlinks,
    linkcolor={red!50!black},
    citecolor={blue!50!black},
    urlcolor={blue!80!black},
}

\usepackage[shortlabels]{enumitem}
\setlist[1]{itemsep=2pt}
\usepackage[capitalise]{cleveref}

\usepackage{autonum} %to number only \eqref'ed equations in a cref compatible way

\newtheorem{theorem}{Theorem}[section]

\newtheorem{lemma}[theorem]{Lemma}

\theoremstyle{remark} 
\newtheorem{remark}[theorem]{Remark}
\newtheorem{example}[theorem]{Example}

\theoremstyle{plain}
\crefname{theoremintro}{Theorem}{Theorems}				% ref name for cleverref package
\crefname{corollaryintro}{Corollary}{Corollary}			% ref name for cleverref package

\title[]{On the tightness of left-invariant contact structures}

\author{Eugenio Bellini$^1$}

\address{$^1$Dipartimento di Matematica Tullio Levi-Civita, Università degli Studi di Padova, Padova, Italy.\newline email: \textbf{\emph{eugenio.bellini@unipd.it}}}

\begin{document}

\maketitle
\begin{abstract}
	We prove that all left-invariant contact structures on three-dimensional Lie groups are tight. The argument is based on Riemannian methods and establishes a unique factorization property for any Lie group admitting a left-invariant contact structure, other than $\mathrm{SU}(2)$. We then make use of such factorization property to construct embeddings of left-invariant contact structures into the standard contact structure on $\mathbb R^3$.
\end{abstract}
\tableofcontents
\section{Statement of the results} 
A \emph{3-dimensional contact manifold} \((M,\xi)\) is a smooth \(3\)-manifold \(M\) equipped with a contact structure \(\xi\), that is, a completely non-integrable plane field \cite{Geiges,Massot-LectureNotes}. A contact manifold is said to be \emph{overtwisted} if it contains an embedded disk whose boundary is tangent to \(\xi\); otherwise, it is called \emph{tight}. This tight/overtwisted dichotomy, originating in the seminal work of Bennequin \cite{bennequin}, represents a turning point in the development of contact topology. While overtwisted contact structures on closed \(3\)-manifolds have been completely classified and tight contact structures have been the subject of extensive study (see for instance \cite{Eliashberg,Elia3,Eliashberg2,Elia4,giroux1,giroux2,HondaClass,VGK2003,SmallSeifert}), determining whether a specific contact structure is tight or overtwisted remains, in general, a non-trivial problem.  

In this work, we address the tightness question in the setting of \emph{left-invariant contact structures}. Recall that a contact structure \(\xi\) on a Lie group \(G\) is called \emph{left-invariant} if it is preserved under left translations, i.e.\ for every \(g \in G\) we have \((L_g)_*\xi = \xi\). Left-invariant contact structures are of independent interest in different areas of mathematics. For example, they appear in contact topology \cite{Geiges1997,Geiges1995}, in higher-dimensional contact and Riemannian geometry \cite{Diatta1,Diatta2}, and in sub-Riemannian geometry \cite{Barilari,Falbel,Boarotto2016}. They also play a role in K-contact geometry and comparison theorems \cite[Sec. 6.3]{ABBR}, as well as in models of the visual cortex \cite{Citti,Cortical_BCGPR}. A classification of three-dimensional Lie groups admitting a left-invariant contact structure can be deduced from the classification of three-dimensional Lie algebras (see, for instance \cite{Barilari}). In particular, there exist infinitely many pairwise non-isomorphic simply connected Lie groups carrying a left-invariant contact structure. 

Our main goal is to determine whether such left-invariant contact structures are tight or overtwisted. We obtain the following result. 

\begin{theorem}\label{thm:left-invariant-tight}
	All left-invariant contact structures on three-dimensional Lie groups are tight.
\end{theorem}

The question of the tightness of left-invariant structures first emerged for the author in the study of contact metric three manifolds. Specifically, the motivation was to investigate whether left-invariant structures could serve as model spaces for Riemannian and sub-Riemannian comparison theories aimed at establishing quantitative tightness results, see \cite{Massot,QDarboux,ABBR}.
As far as the author is aware, no statement concerning the tightness of left-invariant contact structures can be found in the literature.

\subsection{Outline of the proof}

We briefly describe the strategy of the proof in the simply connected case; the non-simply connected case follows because the property of being overtwisted lifts to the universal cover (see, for instance, \cite[Remark\, 2.6]{ABBR}).

Let \((G,\xi)\) be a simply connected Lie group endowed with a left-invariant contact structure. We distinguish two cases according to the Lie algebra \(\mathfrak g\) of \(G\).

\smallskip
\noindent
\textbf{Case 1: \(\mathfrak g \simeq \mathfrak{su}(2)\).}  
Lie's theorem implies that \(G\) is isomorphic to \(\mathrm{SU}(2)\). Moreover, all left-invariant plane fields on \(\mathrm{SU}(2)\) are related by a Lie group isomorphism, and are well known to be tight contact structures (see, for instance, \cite{Massot-LectureNotes}).

\smallskip
\noindent
\textbf{Case 2: \(\mathfrak g \not\simeq \mathfrak{su}(2)\).}  
In this case, the structure of \(G\) is described by the following result.

\begin{theorem}\label{thm:group_v_field_intro}
	Let \((G,\xi)\) be a three-dimensional simply connected Lie group with Lie algebra \(\mathfrak g \not\simeq \mathfrak{su}(2)\), endowed with a left-invariant contact structure $\xi$. Then there exist three \(1\)-dimensional Lie subgroups \(H_1,H_2,H_3\) of $G$, with \(T H_3 \subset \xi\), such that the map
	\[
	p: H_1 \times H_2 \times H_3 \longrightarrow G, \qquad p(h_1,h_2,h_3)=h_1 h_2 h_3
	\]
	is a diffeomorphism.
\end{theorem}
\noindent
Choose generators \(v_i \in \mathfrak{g}\) so that \(T_{1_G}H_i = \mathrm{span}\{v_i\}\) for \(i=1,2,3\). Theorem~\ref{thm:group_v_field_intro} implies that the map
\[
p:\mathbb{R}^3 \longrightarrow G,\qquad 
p(x,y,z)=\exp(xv_1)\exp(yv_2)\exp(zv_3),
\]
where $\exp:\mathfrak g\to G$ is the Lie group exponential map, is a diffeomorphism. In particular, \(G\) is diffeomorphic to \(\mathbb{R}^3\) and admits global coordinates \((x,y,z)\) such that the vector field \(\partial_z\) is tangent to $\xi$, since $v_3\in \xi$.
Applying the following lemma, proved in \cref{sec:proof_of_main}, immediately yields tightness of \(\xi\).

\begin{lemma}
	Let \(\xi\) be a contact structure on \(\mathbb{R}^3\). If there exist global coordinates \((x,y,z)\) such that the vector field \(\partial_z\) is tangent to $\xi$, then \(\xi\) is a tight contact structure.
\end{lemma}

It remains to justify Theorem~\ref{thm:group_v_field_intro}.  
For $G = \widetilde{\mathrm{SL}}(2)$, the factorization is provided explicitly.  
For all other groups with \(\mathfrak g \not\simeq \mathfrak{su}(2)\), we first construct a left-invariant Riemannian metric on $G$ together with a $2$-dimensional subgroup $H$ and a totally geodesic subgroup $H_3$, which is tangent to $\xi$ and orthogonal to $H$.  
We then apply the following lemma.

\begin{lemma}\label{lem:orthogonal_is_product}
	Let $G$ be a simply connected Lie group endowed with a left-invariant Riemannian metric. Assume the existence of an immersed codimension-one subgroup $H$ and a totally geodesic subgroup $H_3$, orthogonal to each other. Then the map
	\[
	p:H \times H_3 \longrightarrow G, \qquad p(h,h_3)=hh_3
	\]
	is a diffeomorphism.
\end{lemma}
\noindent 
Since $H$ is $2$-dimensional, the classification of $2$-dimensional Lie groups yields a further factorization \(H \simeq H_1 \times H_2 \) (cf. \cref{sec:proof_group_field}), where $\simeq$ denotes a diffeomorphism of smooth manifolds (not a Lie group isomorphism).

\section{Two lemmas on Riemannian Lie groups}\label{4sec:Hopf-Rinow}
	Given a Riemannian manifold $(M,\eta)$, we denote the normal bundle of an immersed submanifold $S\subset M$ as $TS^{\perp}$.
We say that $S$ is complete as a metric subspace of $M$ if $(S,d|_S)$ is a complete metric space, $d|_S$ being the restriction of the Riemannian distance. 
\begin{theorem}\label{4thm:Hopf-Rinow}
    Let $(M,\eta)$ be a Riemannian manifold and $S$ be an immersed submanifold which is complete as a metric subspace. The following are equivalent:
    \begin{itemize}
        \item[i)] $(M,\eta)$ is a complete Riemannian manifold,
        \item[ii)] The normal exponential map $\exp^{\perp}:TS^\perp\to M$, i.e., the restriction of the Riemannian exponential map to the normal bundle $TS^\perp$, is well-defined on the whole $TS^\perp$.
    \end{itemize}
\end{theorem}
\begin{remark}
    Notice that for $S=\{p\}$ we recover the classical statement of Hopf-Rinow theorem.
\end{remark}
The proof of \cref{4thm:Hopf-Rinow} is analogous to the one of \cite[Thm.\ 2.8, Chap.\ 7]{DoCarmo}, replacing the normal neighborhood of a point with a normal tubular neighborhood around $S$. 
\begin{lemma}\label{4cor:groups}
    Let $G$ be a Lie group endowed with a left-invariant Riemannian metric. Let $H\subset G$ be an immersed Lie subgroup. If the normal exponential map $\exp^{\perp}:TH^\perp\to G$ is an immersion then it is a covering map.
\end{lemma}
\begin{proof}
    We denote with $\eta$ the left-invariant Riemannian metric. Since the normal exponential map is an immersion, $(TH^\perp,(\exp^{\perp})^*\eta)$ is a Riemannian manifold. 
    Notice that $H$ acts on $TH^\perp$:
    \begin{equation}\label{eq:H-action}
        \begin{aligned}
            H\times TH^\perp\to TH^\perp,\qquad 
            (h,v)\mapsto dL_{h}v,
        \end{aligned}
    \end{equation}
    where $L_h:G\to G$ denotes the left translation $L_h(g)=hg$.
    The action preserves the metric $(\exp^{\perp})^*\eta$
    \begin{equation}
        dL_h^*(\exp^{\perp})^*\eta=(\exp^\perp\circ\,dL_h)^*\eta=(L_h\circ\exp^\perp)^*\eta=(\exp^{\perp})^* L_h^*\eta=(\exp^{\perp})^*\eta.
    \end{equation}
    The zero section $s_0\subset TH^\perp$, is invariant under the action \eqref{eq:H-action}. Furthermore, $H$ acts transitively on $s_0$. This implies that $s_0\subset TH^\perp$ is complete as a metric subspace of $(TH^\perp,(\exp^{\perp})^*\eta)$, because it is  a locally compact metric space with transitive isometry group.
    Moreover, the normal exponential map of the zero section of $TH^\perp$, which we denote $\exp^{\perp}_0:\left(Ts_0\right)^\perp\to TH^\perp$, is well-defined, the normal geodesics being the 1-dimensional subspaces of the fibres of the vector bundle $TH^\perp$. Theorem \ref{4thm:Hopf-Rinow} implies that $(TH^\perp,(\exp^{\perp})^*\eta)$ is a complete Riemannian manifold. Therefore
    \begin{equation}
        \exp^{\perp}:(TH^\perp,(\exp^{\perp})^*\eta)\to (G,\eta)
    \end{equation}
    is a surjective local isometry of complete Riemannian manifolds and thus, according to \cite[Lem.\ 3.3, Chap.\ 7]{DoCarmo}, a covering map.
\end{proof}
\begin{lemma}\label{lem:orthogonal_groups}
	Let \(G\) be a simply connected Lie group endowed with a left-invariant Riemannian metric. Assume the existence of an immersed codimension-one subgroup \(H\) and a totally geodesic subgroup \(H_3\), orthogonal to each other. Then the map
	\[
	p:H \times H_3 \longrightarrow G, \qquad p(h,h_3)=hh_3
	\]
	is a diffeomorphism.
\end{lemma}

\begin{proof}
	Let \(X\) be a left-invariant vector field spanning \(TH_3\). Since \(H_3\) is totally geodesic and orthogonal to \(H\), the vector field \(X\) is geodesic and orthogonal to \(H\). Therefore, the normal exponential map can be written as
	\[
	\exp^\perp: TH^{\perp} \simeq H \times \mathbb{R} \to G, \qquad \exp^\perp(h,t) = e^{tX}(h).
	\]
	Because \(X\) and \(H\) are transverse, \(\exp^\perp\) is an immersion. By \cref{4cor:groups}, it is also a covering map. Since \(G\) is simply connected, \(\exp^\perp\) is a diffeomorphism. It follows that the map
	\[
	\psi: \mathbb{R} \to H_3, \qquad \psi(t) =\exp^\perp(1_G,t)=e^{tX}(1_G),
	\]
	where $1_G$ denotes the identity element of $G$, is a diffeomorphism as well. Finally, the map \(p: H \times H_3 \to G\) can be expressed as 
	\[
	p(h,h_3)=hh_3=e^{\psi^{-1}(h_3)X}(h)=\exp^{\perp}(h,\psi^{-1}(h_3)),
	\]
	where the second equality follows from left-invariance of $X$ and definition of $\psi$. Since $\exp^{\perp}$ and $\psi$ are diffeomorphisms, we conclude the proof.
\end{proof}

\begin{lemma}\label{lem:geodesics}
    Let $(G,\eta)$ be a Lie group with a left-invariant Riemannian metric $\eta$ and let $X_1,\dots,X_n$ be a left-invariant orthonormal frame with structure constants  $c_{ij}^k$: 
    \begin{equation}
        [X_i,X_j]=\sum_{k=1}^nc_{ji}^kX_k,\qquad i,j=1,\dots,n.
    \end{equation}
    Then, the vector field $X_1$ is geodesic if and only if $c_{1j}^1=0$ for all $j=1,\dots,n$.
\end{lemma}
\begin{proof}
    Let $\nabla$ denote the Levi-Civita connection. For a fixed $j\in\{1,\dots,n\}$ we compute 
    \begin{equation}
        \eta(\nabla_{X_1}X_1,X_j)=-\eta(X_1,\nabla_{X_1}X_j)=-\eta(X_1,\nabla_{X_j}X_1+[X_1,X_j])=\eta(X_1,[X_j,X_1])=c_{1j}^1.
    \end{equation}
    Therefore $\nabla_{X_1}X_1=0$ if and only if $c_{1j}^1=0$ for all $j=1,\dots,n$.
\end{proof}
\section{A unique factorization property: proof of \cref{thm:group_v_field_intro}}\label{4sec:groups}
In this section we show that any contact group with $\mathfrak g\not\simeq \mathfrak{su}(2) $ satisfies a unique factorization property. We illustrate this fact in \cref{ex:st_sl2}, and, after introducing a convenient basis for $\mathfrak g$ in \cref{lem:can_frame}, we prove \cref{thm:group_v_field_intro}.
\begin{example}\label{ex:st_sl2}
    Let $\widetilde{\mathrm{SL}}(2)$ be the universal cover of $\mathrm{SL}(2)$, the latter being the group of $2\times 2$ real matrices with determinant $1$.  A basis for its Lie algebra $\mathfrak{sl}(2)$ is given by \begin{equation}\label{eq:sl2_basis}
        v_1 = \frac{1}{2}\begin{pmatrix} 1 & 0 \\ 0 & -1 \end{pmatrix}, \qquad
v_2 =\frac{1}{2}\begin{pmatrix} 0 & 1 \\ 1 & 0 \end{pmatrix}, \qquad
v_0=\frac{1}{2}\begin{pmatrix} 0 & -1 \\ 1 & 0 \end{pmatrix}.
    \end{equation}
	Note that such basis of $\mathfrak{sl}(2)$ satisfies
	\begin{equation}
		[v_1,v_0]=-v_2,\qquad 
		[v_2,v_0]=v_1,\qquad
		[v_2,v_1]=v_0.
	\end{equation}
    We define the standard contact structure on $\widetilde{\mathrm{SL}}(2)$ as $\xi=\mathrm{span}\{v_1,v_2\}$, extended to the whole group by left translations. Given $A\in \mathrm{SL}(2)$, seeing its columns as vectors $(x,y)^T\in\mathbb R^2$, there exists a unique $O\in \mathrm{SO}(2)$ mapping the first column of $A$ to a vector belonging to the positive $x$-axes. That is to say, for any $A\in \mathrm{SL}(2)$ there exists a unique $O\in \mathrm{SO}(2)$ such that $OA \in H$ where
    \begin{equation}\label{eq:non_comm_2}
        H=\left\{\begin{pmatrix}
        	a_1 & a_2\\
        	0 & a_3
        \end{pmatrix}\in \mathrm{SL}(2)\,:\,a_i\in\mathbb R,\,a_{1}>0\right\}=\left\{\begin{pmatrix}
             e^a & b \\
             0 & e^{-a}
         \end{pmatrix}\in \mathrm{SL}(2)\,:\,a,b\in\mathbb R\right\}.
    \end{equation}
    Therefore any matrix $A\in \mathrm{SL}(2)$ can be uniquely factorized as $A=O^{-1}B$, with $O\in \mathrm{SO}(2)$ and $B\in H$. Another way to state this unique factorization property is saying that the following map 
\begin{equation}\label{eq:p_sl2}
    \mathrm{SO}(2)\times H\to \mathrm{SL}(2),\qquad (O,B)\mapsto OB,
    \end{equation}
    is a diffeomorphism. Introducing the following subgroups  
    \begin{equation}\label{eq:H_3}
        H_2=\left\{\begin{pmatrix}
             1 & v \\
             0 & 1
         \end{pmatrix}\in \mathrm{SL}(2)\,:\,v\in\mathbb R\right\},\qquad H_3=\left\{\begin{pmatrix}
             e^u & 0 \\
             0 & e^{-u}
         \end{pmatrix}\in \mathrm{SL}(2)\,:\,u\in\mathbb R\right\},
    \end{equation}
   one can check that the map $H_2\times H_3\to H$ defined as $(h_2,h_3)\mapsto h_2h_3$ is a diffeomorphism. We can rewrite the diffeomorphism \eqref{eq:p_sl2} as 
    \begin{equation}
        \mathrm{SO}(2)\times H_2\times H_3\to \mathrm{SL}(2),\qquad (O,h_2,h_3)\mapsto Oh_2h_3.
    \end{equation}
    Passing to the universal covers, denoting with $H_1$ the universal cover of $\mathrm{SO}(2)$, we find the following diffeomorphism
    \begin{equation}\label{eq:real_p_sl2}
        p:H_1\times H_2\times H_3\to \widetilde{\mathrm{SL}}(2),\qquad p(h_1,h_2,h_3)=h_1h_2h_3.
    \end{equation}
    Notice that $TH_3\subset \xi$. Indeed, from \eqref{eq:sl2_basis} and the definition of $H_3$ in \eqref{eq:H_3}, we deduce that  $TH_3=\mathrm{span}\{v_1\}$, while, by definition, $\xi=\mathrm{span}\{v_1,v_2\}$. 
    \end{example}
\begin{lemma}\label{lem:can_frame}
    Let $(G,\xi)$ be a three-dimensional Lie group with a left-invariant contact structure. Let $\mathfrak g$ denote its Lie algebra. Then there exists a basis $\{v_0,v_1,v_2\}$ for $\mathfrak g$ such that $v_1,v_2\in\xi$, and whose Lie brackets satisfy
    \begin{equation}\label{eq:can_frame}
[v_0,v_1]=c_{10}^2v_2,\qquad 
[v_0,v_2]=c_{20}^1v_1,\qquad
[v_1,v_2]=c_{21}^1v_1+c_{21}^2v_2-v_0,
\end{equation}
for some $c_{ij}^k\in \mathbb R$. Furthermore, if $c_{10}^2=c_{20}^1=0$, then we can choose $v_1,v_2\in\xi$ so that 
\begin{equation}\label{eq:heis_cns}
    [v_0,v_1]=[v_0,v_2]=0,\qquad [v_1,v_2]=c_{21}^2v_2-v_0.
\end{equation}
\end{lemma}
\begin{proof}
    We fix a left-invariant contact form $\alpha$, i.e., a left-invariant differential form such that $\xi=\ker\alpha$. Let $R$ denote the associated Reeb vector field, i.e., the unique vector field satisfying 
    \begin{equation}
        \alpha(R)=1,\qquad d\alpha(R,\cdot)=0.
    \end{equation}
    Let $\{w_0,w_1,w_2\}$ be a left-invariant trivialization of $TG$ satisfying
    \begin{equation}\label{eq:adapp_frame}
        w_0=R,\qquad w_1,w_2\in\xi.
    \end{equation}
    Since the flow of the Reeb vector field preserves the contact structure, we have the endomorphism
    \begin{equation}\label{eq:end_w_0}
        \mathrm{ad}\,w_0:\xi\to\xi,\qquad v\mapsto[w_0,v].
    \end{equation}
    Assume first that $\mathrm{ad}\,w_0$ is identically zero. Then we have constants $a_{21}^1,a_{21}^2,a_{21}^0\in\mathbb R$ such that 
    \begin{equation}
        [w_0,w_1]=[w_0,w_2]=0,\qquad [w_1,w_2]=a_{21}^1w_1+a_{21}^2w_2+a_{21}^0w_0.
    \end{equation}
    If $a_{21}^1=a_{21}^2=0$, then the basis satisfying \eqref{eq:heis_cns} is obtained setting $v_1=w_1,\,v_2=w_2,\,v_0=-a_{21}^0w_0$.
    Otherwise, if for instance $a_{21}^2\neq 0$, then the basis $\{v_0,v_1,v_2\}$ for $\mathfrak g$, with $v_1,v_2\in\xi$, defined by 
    \begin{equation}
     v_1=\frac{1}{a_{21}^2}w_1,\qquad v_2=a_{21}^1w_1+a_{21}^2w_2,\qquad v_0=-a_{21}^0w_0,
    \end{equation}
    satisfies \eqref{eq:heis_cns}. The case $a_{21}^1\neq 0$ is analogous.
    
    Assume that \eqref{eq:end_w_0} is not identically zero. We claim that, nonetheless, $\mathrm{ad}\,w_0$ is traceless. Let $\theta_0,\theta_1,\theta_2$ be a trivialization of $T^*G$ dual to $w_0,w_1,w_2$, i.e., $\theta_i(w_j)=\delta_{ij}$. Being $w_1,w_2$ a basis for $\xi$, which is a contact distribution, the bracket $[v_1,v_2]$ is transverse to $\xi$. One has $\theta_0([w_1,w_2])\neq 0$. Up to re-scaling $w_0$ we may assume the existence of $c_{21}^1,c_{21}^2\in\mathbb R$ such that 
 	\begin{equation}\label{eq:normalized_frame}
 		[w_1,w_2]=c_{21}^1w_1+c_{21}^2w_2-w_0.
 	\end{equation}
 	Since $\theta_i(w_j)=\delta_{ij}$, an application of Cartan formula $d\theta_i(v,w)=v(\theta_i(w))-w(\theta_i(v))-\theta_i([v,w])$ yields
 	\begin{equation}
 		d\theta_k=-\theta_{k}([w_0,w_1])\theta_0\wedge\theta_1-\theta_{k}([w_0,w_2])\theta_0\wedge\theta_2-\theta_{k}([w_1,w_2])\theta_1\wedge\theta_2.
 	\end{equation}
 	Therefore, in light of \eqref{eq:normalized_frame}, for $k=0$ one finds 
 	\begin{equation}
 		d\theta_0=\theta_1\wedge \theta_2,
 	\end{equation}
 	where we have used that $\theta_0([w_0,w_1])=\theta_0([w_0,w_2])=0$, since $\mathrm{ad}\,w_0$ maps $\xi=\mathrm{span}\{w_1,w_2\}$ to itself.
 	From the identity $d^2\theta_0=0$ we deduce
 	\begin{equation}
 		\begin{aligned}
 			0&=d^2\theta_0
 			=d\theta_1\wedge\theta_2-\theta_1\wedge d\theta_2=-(\theta_1([w_0,w_1])+\theta_2([w_0,w_2]))\theta_0\wedge \theta_1\wedge \theta_2\\
 			&=-\mathrm{trace}(\mathrm{ad}\,w_0)\theta_0\wedge \theta_1\wedge \theta_2.
 		\end{aligned}
 	\end{equation}
 	Hence $\mathrm{ad}\,w_0:\xi\to \xi$ is traceless. Therefore, there exists a basis $v_1,v_2$ for $\xi$ and real numbers $c_{10}^2, c_{20}^1$ such that 
 	\begin{equation}
 		[w_0,v_1]=c_{10}^2v_2,\qquad [w_0,v_2]=c_{20}^1v_1.
 	\end{equation} 
 	Setting $v_0=w_0$, the above equation and \eqref{eq:normalized_frame} conclude the proof.
    %\todo{Non c'è un modo più easy di dimostrare traceless?} Indeed let $a_{ij}^k$ be the structure constants of $w_0,w_1,w_2$, i.e., 
    %\begin{equation}\label{eq:cnst_vec}
   % [w_i,w_j]=\sum_{k=0}^2a_{ji}^kw_k,\qquad i,j=0,1,2,
    %\end{equation}
    %and let $\theta_0,\theta_1,\theta_2$ be a trivialization of $T^*G$ dual to $w_0,w_1,w_2$, i.e., $\theta_i(w_j)=\delta_{ij}$. 
    %From 
    %\eqref{eq:adapp_frame} and \eqref{eq:cnst_vec} we deduce that 
    %\begin{equation}
    %\theta_0=\alpha,\qquad d\theta_0=-\theta_0([w_1,w_2])\theta_1\wedge\theta_2=a_{12}^0\theta_1\wedge\theta_2.
    %\end{equation}
    %Since $\theta_0=\alpha$, and $d\alpha\neq 0$, then $a_{12}^0\neq 0$. Up to rescaling $\theta_1$ we may assume $a_{12}^0=1$. Exploiting the identity $d^2\theta_0=0$ we get 
    %\begin{equation}
    %0=d^2\theta_0=d\theta_1\wedge\theta_2-\theta_1\wedge d\theta_2=\left(a_{01}^1+a_{02}^2\right)\theta_0\wedge\theta_1\wedge\theta_2
    %=\mathrm{trace}(\mathrm{ad}\,w_0)\theta_0\wedge\theta_1\wedge\theta_2,
    %\end{equation}
    %where the third equality follows from $\theta_i(w_j)=\delta_{ij}$ and \eqref{eq:cnst_vec}. We deduce that the endomorphism \eqref{eq:end_w_0} is traceless. It follows that there exists a basis $v_1,v_2$ for $\xi$ and real numbers $c_{10}^2, c_{20}^1$ such that 
    %\begin{equation}
    %[w_0,v_1]=c_{10}^2v_2,\qquad [w_0,v_2]=c_{20}^1v_1.
    %\end{equation} 
    %Up to rescaling $w_0$ we can assume $\theta_0([v_1,v_2])=-1$, therefore the basis $v_0:=w_0, v_1,v_2$ satisfies equation \eqref{eq:st_const} for some constants $c_{ij}^k$.

\end{proof}
Before proving \cref{thm:group_v_field_intro}, we introduce some terminology. We call the couple $(G,\xi)$, where $G$ is a Lie group and $\xi$ a left-invariant contact structure, a contact group. We say that two contact groups $(G,\xi)$ and $(G',\xi')$ are isomorphic if there exists a group isomorphism $\varphi:G\to G'$ such that $\varphi_*\xi=\xi'$.
\subsection{Proof of \cref{thm:group_v_field_intro}}\label{sec:proof_group_field}
\begin{proof}
Having already discussed the case of \cref{ex:st_sl2}, we assume throughout the rest of the proof that $(G,\xi)$ is not isomorphic to the contact group described there.

According to \cref{lem:can_frame} there exist a basis of left-invariant vector fields $\{v_0,v_1,v_2\}$, with $v_1,v_2\in\xi$, and real numbers $c_{ij}^k$ such that 
\begin{equation}\label{eq:st_const}
	[v_1,v_0]=c_{01}^2v_2,\qquad 
	[v_2,v_0]=c_{02}^1v_1,\qquad
	[v_1,v_2]=c_{21}^1v_1+c_{21}^2v_2-v_0.
\end{equation}
Let $\{\theta_0,\theta_1,\theta_2\}$ be the trivialization of $T^*G$ dual to $v_0,v_1,v_2$, i.e., $\theta_i(v_j)=\delta_{ij}$. We define a left-invariant Riemannian metric $\eta$ by 
\begin{equation}
	\eta=\theta_0\otimes \theta_0+\theta_1\otimes\theta_1+\theta_2\otimes\theta_2.
\end{equation}
Note that $v_0,v_1,v_2$ is an orthonormal frame. 

\noindent
\textbf{Claim.} There exist a 2-dimensional subalgebra $\mathfrak h\subset \mathfrak g$ and a left-invariant geodesic vector field $X$ tangent to $\xi$, which are orthogonal to each other at the identity. 

\noindent
Since $\theta_i(v_j)=\delta_{ij}$, an application of Cartan formula $d\theta_i(v,w)=v(\theta_i(w))-w(\theta_i(v))-\theta_i([v,w])$ yields
\begin{equation}
	d\theta_0=\theta_1\wedge \theta_2,\qquad d\theta_1=c_{02}^1\theta_0\wedge\theta_2+c_{12}^1\theta_1\wedge\theta_2,\qquad d\theta_2=c_{01}^2\theta_0\wedge\theta_1+c_{12}^2\theta_1\wedge\theta_2.
\end{equation} 
From the identities $d^2\theta_1=d^2\theta_2=0$ we deduce the constraints
    \begin{equation}
        c_{01}^2c_{12}^1=0,\qquad c_{02}^1c_{12}^2=0.
    \end{equation}
    We have the following four possibilities.
    \begin{enumerate}
        \item[$i$.] If $c_{01}^2=c_{12}^2=0$, then the structural equations \eqref{eq:st_const} of $\mathfrak g$ reduces to 
        \begin{equation}
[v_1,v_0]=0,\qquad 
[v_2,v_0]=c_{02}^1v_1,\qquad
[v_1,v_2]=c_{21}^1v_1-v_0.
    \end{equation}
    We set $\mathfrak h=\mathrm{span}\{v_1,v_0\}$, $X=v_2$. Notice that $\mathfrak h$ is a sub-algebra and that $X$ is tangent to $\xi$ and orthogonal to $\mathfrak h$. Moreover, since $c_{2j}^2=0$ for $j=0,1,2$, \cref{lem:geodesics} implies that $X$ is geodesic.
    \item[$ii$.] If $c_{02}^1=c_{12}^1=0$,  then the structural equations \eqref{eq:st_const} of $\mathfrak g$ reduces to 
        \begin{equation}
[v_1,v_0]=c_{01}^2v_2,\qquad 
[v_2,v_0]=0,\qquad
[v_1,v_2]=c_{21}^2v_2-v_0.
    \end{equation}
    We set $\mathfrak h=\mathrm{span}\{v_2,v_0\}$, $X=v_1$. Notice that $\mathfrak h$ is a sub-algebra and that $X$ is tangent to $\xi$ and orthogonal to $\mathfrak h$. Moreover, since $c_{1j}^1=0$ for $j=0,1,2$, \cref{lem:geodesics} implies that $X$ is geodesic.
    \item[$iii$.] If $c_{01}^2=c_{02}^1=0$, according \cref{lem:can_frame}, we can choose $v_0,v_1,v_2$ satisfying \eqref{eq:heis_cns}. We set $\mathfrak{h}=\mathrm{span}\{v_2,v_0\}$, $X=v_1$. Again, $\mathfrak h$ is a sub-algebra, $X$ is tangent to $\xi$ and orthogonal to $\mathfrak h$. Moreover, since $c_{2j}^2=0$ for $j=0,1,2$, \cref{lem:geodesics} implies that $X$ is geodesic.
    \item[$iv$.] If $c_{21}^1=c_{21}^2=0$, and all other structure constants are nonzero, then up to a change of basis for $\xi$ and a re-scaling of $v_0$ we can assume the existence of a constant $\lambda\in\{+1,-1\}$ such that 
    \begin{equation}\label{eq:kappa}
        [v_1,v_0]=\lambda v_2,\qquad 
[v_2,v_0]=-\lambda v_1,\qquad
[v_2,v_1]=v_0.
    \end{equation}
    If $\lambda=1$, we recognize the bracket relations of $\mathfrak {su}(2)$, contradicting the hypothesis of the theorem. Similarly, if $\lambda=-1$, comparing the brackets of \eqref{eq:sl2_basis} with \eqref{eq:kappa}, we deduce that $(G,\xi)$ is isomorphic to the contact group of \cref{ex:st_sl2}, contradicting the assumption made at the beginning of the proof.
 \end{enumerate}   
 The claim is proved. Let $H,H_3$ be the immersed subgroups corresponding to the subalgebras $\mathfrak h$ and $\mathrm{span}\{X_{1_G}\}$ respectively. By construction $H$ is 2-dimensional and $H_3$ is totally geodesic, tangent to $\xi$ and orthogonal to $H$. Applying \cref{lem:orthogonal_groups}, we deduce that the map 
\begin{equation}\label{eq:q_diffeo}
   q:H\times H_3\to G,\qquad q(h,h_3)=hh_3,
\end{equation}
is a diffeomorphism. Because $G$ is simply connected, $H$ is a 2-dimensional simply connected Lie group. There exist exactly two simply connected 2-dimensional Lie groups. One is commutative, the other one is described in equation \eqref{eq:non_comm_2}. Both of them contain two 1-dimensional subgroups $H_1,H_2\subset H$ such that the map 
\begin{equation}\label{eq:H_diffeo}
    H_1\times H_2\to H,\qquad (h_1,h_2)\mapsto h_1 h_2,
\end{equation}
is a diffeomorphism. If $H$ is commutative, then $H$ is isomorphic to $\mathbb R^2$ and this fact is clear. In the non commutative case the two 1-dimensional subgroups of \eqref{eq:non_comm_2} for which \eqref{eq:H_diffeo} holds are exhibited in equation \eqref{eq:H_3}. Finally combining \eqref{eq:H_diffeo} and \eqref{eq:q_diffeo} we define
\begin{equation}
    p:H_1\times H_2\times H_3\to G,\qquad p(h_1,h_2,h_3):=q(h_1h_2,h_3)=h_1 h_2h_3.
\end{equation}
Since both \eqref{eq:H_diffeo} and \eqref{eq:q_diffeo} are diffeomorphisms, we deduce that $p:H_1\times H_2\times H_3\to G$ is a diffeomorphism. 
    \end{proof}
\section{Proof of the main result}\label{sec:proof_of_main}
In this section we prove the main result. We need one preliminary lemma. 
\begin{lemma}\label{lem:rectifiable_tight}
	Let \(\xi\) be a contact structure on \(\mathbb{R}^3\). If there exist global coordinates \((x,y,z)\) such that the vector field \(\partial_z\) is tangent to $\xi$, then \(\xi\) is a tight contact structure.
\end{lemma}

\begin{proof}
	Let \(\alpha\) be a 1-form defining \(\xi\), so that \(\xi = \ker \alpha\). Since \(\partial_z\) is tangent to $\xi$, we have \(\alpha(\partial_z) \equiv 0\), and therefore
	\[
	\alpha = \alpha(\partial_x)\, dx + \alpha(\partial_y)\, dy.
	\]
	Because \(\alpha\) is a contact form, it never vanishes. Hence, there exists a smooth function \(f\) such that
	\[
	\alpha = \sqrt{\alpha(\partial_x)^2 + \alpha(\partial_y)^2} \, (\cos f \, dx + \sin f \, dy).
	\]
	The contact condition \(\alpha \wedge d\alpha \neq 0\) gives
	\[
	\alpha \wedge d\alpha = -(\alpha(\partial_x)^2 + \alpha(\partial_y)^2) (\partial_z f) \, dx \wedge dy \wedge dz \neq 0,
	\]
	so \(\partial_z f \neq 0\) everywhere. This implies that the map
	\[
	\phi: \mathbb{R}^3 \to \mathbb{R}^3, \qquad \phi(x,y,z) = (u,v,w) := (x,y,f(x,y,z))
	\]
	is a diffeomorphism onto its image. By construction,
	\[
	\phi_* \ker \alpha = \ker \{ \cos w \, du + \sin w \, dv \}.
	\]
	The contact structure \(\ker \{ \cos w \, du + \sin w \, dv \}\) is the standard contact structure on \(\mathbb{R}^3\), which is known to be tight \cite{Massot-LectureNotes}. Therefore, \(\xi\) is tight as well.
\end{proof}
\begin{theorem}
	All left-invariant contact structures on three-dimensional Lie groups are tight.
\end{theorem}
\begin{proof}
	Since the property of being overtwisted lifts to the universal cover (see for instance \cite[Remark\, 2.6]{ABBR}), we may assume without loss of generality that \(G\) is simply connected.
	
	Assume $\mathfrak g \simeq \mathfrak{su}(2)$. Since $G$ is simply connected, $G \simeq \mathrm{SU}(2)$. The automorphism group of $\mathfrak{su}(2)$ acts transitively on planes. Therefore, up to group isomorphism, $\mathrm{SU}(2)$ carries a unique left-invariant plane field, which is well known to be a tight contact structure \cite{Massot-LectureNotes}.
	
	Next, consider the case  \(\mathfrak g \not\simeq \mathfrak{su}(2)\), and let
	\[
	p: H_1 \times H_2 \times H_3 \to G
	\]
	be the diffeomorphism provided by Theorem~\ref{thm:group_v_field_intro}. Since \(G\) is simply connected, each subgroup \(H_i\) is diffeomorphic to \(\mathbb{R}\). Choosing generators \(v_i \in T_e H_i\), the maps
	\[
	\mathbb{R} \to H_i, \qquad t \mapsto \exp(t v_i)
	\]
	are diffeomorphisms. It follows that the map
	\[
	p: \mathbb{R}^3 \to G, \qquad p(x,y,z) = \exp(x v_1) \exp(y v_2) \exp(z v_3)
	\]
	is also a diffeomorphism. Observe that the vector field \(\partial_z\) is tangent to $\xi$. Indeed, let \(X\) denote the left-invariant extension of \(v_3\). Since $X$ is left invariant, its flow acts by right translations. In particular
	\[
	p(x,y,z) = e^{zX}(\exp(x v_1) \exp(y v_2)),
	\]
	which implies \(p_* \partial_z = X \in \xi\). Applying Lemma~\ref{lem:rectifiable_tight}, we conclude that \(\xi\) is tight.
\end{proof}
\bibliographystyle{alphaabbr}
	\bibliography{biblio}

@article {Diatta2,
    AUTHOR = {Diatta, Andr\'{e}},
     TITLE = {Riemannian geometry on contact {L}ie groups},
   JOURNAL = {Geom. Dedicata},
  FJOURNAL = {Geometriae Dedicata},
    VOLUME = {133},
      YEAR = {2008},
     PAGES = {83--94},
      ISSN = {0046-5755},
   MRCLASS = {53D10 (53C25 53C50)},
  MRNUMBER = {2390070},
MRREVIEWER = {Alberto Medina},
       DOI = {10.1007/s10711-008-9236-2},
       URL = {https://doi.org/10.1007/s10711-008-9236-2},
}

@article {Diatta1,
    AUTHOR = {Diatta, Andr\'{e}},
     TITLE = {Left invariant contact structures on {L}ie groups},
   JOURNAL = {Differential Geom. Appl.},
  FJOURNAL = {Differential Geometry and its Applications},
    VOLUME = {26},
      YEAR = {2008},
    NUMBER = {5},
     PAGES = {544--552},
      ISSN = {0926-2245},
   MRCLASS = {53D10 (17B30)},
  MRNUMBER = {2458280},
MRREVIEWER = {Michel Goze},
       DOI = {10.1016/j.difgeo.2008.04.001},
       URL = {https://doi.org/10.1016/j.difgeo.2008.04.001},
}

@article {Citti,
    AUTHOR = {Citti, G. and Sarti, A.},
     TITLE = {A cortical based model of perceptual completion in the
              roto-translation space},
   JOURNAL = {J. Math. Imaging Vision},
  FJOURNAL = {Journal of Mathematical Imaging and Vision},
    VOLUME = {24},
      YEAR = {2006},
    NUMBER = {3},
     PAGES = {307--326},
      ISSN = {0924-9907},
   MRCLASS = {92C20 (53A05)},
  MRNUMBER = {2235475},
       DOI = {10.1007/s10851-005-3630-2},
       URL = {https://doi.org/10.1007/s10851-005-3630-2},
}

@article {Elia4,
	AUTHOR = {Borman, Matthew Strom and Borman, Matthew Strom and
	Eliashberg, Yakov and Murphy, Emmy},
	TITLE = {Existence and classification of overtwisted contact structures
	in all dimensions},
	JOURNAL = {Acta Math.},
	FJOURNAL = {Acta Mathematica},
	VOLUME = {215},
	YEAR = {2015},
	NUMBER = {2},
	PAGES = {281--361},
	ISSN = {0001-5962},
	MRCLASS = {57R17 (53D10)},
	MRNUMBER = {3455235},
	MRREVIEWER = {Yang Huang},
	DOI = {10.1007/s11511-016-0134-4},
	URL = {https://doi.org/10.1007/s11511-016-0134-4},
}

@incollection {VGK2003,
	AUTHOR = {Colin, Vincent and Giroux, Emmanuel and Honda, Ko},
	TITLE = {On the coarse classification of tight contact structures},
	BOOKTITLE = {Topology and geometry of manifolds ({A}thens, {GA}, 2001)},
	SERIES = {Proc. Sympos. Pure Math.},
	VOLUME = {71},
	PAGES = {109--120},
	PUBLISHER = {Amer. Math. Soc., Providence, RI},
	YEAR = {2003},
	MRCLASS = {53D35 (57R17)},
	MRNUMBER = {2024632},
	MRREVIEWER = {John B. Etnyre},
	DOI = {10.1090/pspum/071/2024632},
	URL = {https://doi.org/10.1090/pspum/071/2024632},
}

@article {Falbel,
	AUTHOR = {Falbel, Elisha and Gorodski, Claudio},
	TITLE = {Sub-{R}iemannian homogeneous spaces in dimensions {$3$} and
	{$4$}},
	JOURNAL = {Geom. Dedicata},
	FJOURNAL = {Geometriae Dedicata},
	VOLUME = {62},
	YEAR = {1996},
	NUMBER = {3},
	PAGES = {227--252},
	ISSN = {0046-5755},
	MRCLASS = {53C30},
	MRNUMBER = {1406439},
	MRREVIEWER = {Fernand Pelletier},
	DOI = {10.1007/BF00181566},
	URL = {https://doi.org/10.1007/BF00181566},
}

@article{Boarotto2016,
	author = {Francesco Boarotto},
	title = {Conformal Equivalence of 3D Contact Structures on Lie Groups},
	journal = {Journal of Dynamical and Control Systems},
	volume = {22},
	number = {2},
	pages = {251--283},
	year = {2016},
	doi = {10.1007/s10883-015-9273-8},
	url = {https://link.springer.com/article/10.1007/s10883-015-9273-8}
}

@article {HondaClass,
	AUTHOR = {Honda, Ko},
	TITLE = {On the classification of tight contact structures. {II}},
	JOURNAL = {J. Differential Geom.},
	FJOURNAL = {Journal of Differential Geometry},
	VOLUME = {55},
	YEAR = {2000},
	NUMBER = {1},
	PAGES = {83--143},
	ISSN = {0022-040X},
	MRCLASS = {53D35 (57M50 57R17)},
	MRNUMBER = {1849027},
	MRREVIEWER = {Hansj\"{o}rg Geiges},
	URL = {http://projecteuclid.org/euclid.jdg/1090340567},
}

@article {SmallSeifert,
	AUTHOR = {Ghiggini, Paolo and Ghiggini, Paolo and Lisca, Paolo and
	Stipsicz, Andr\'{a}s I.},
	TITLE = {Tight contact structures on some small {S}eifert fibered
	3-manifolds},
	JOURNAL = {Amer. J. Math.},
	FJOURNAL = {American Journal of Mathematics},
	VOLUME = {129},
	YEAR = {2007},
	NUMBER = {5},
	PAGES = {1403--1447},
	ISSN = {0002-9327},
	MRCLASS = {57R17 (57N10)},
	MRNUMBER = {2354324},
	MRREVIEWER = {Danny C. Calegari},
	DOI = {10.1353/ajm.2007.0033},
	URL = {https://doi.org/10.1353/ajm.2007.0033},
}

@article {Elia3,
    AUTHOR = {Eliashberg, Yakov},
     TITLE = {Contact {$3$}-manifolds twenty years since {J}. {M}artinet's
              work},
   JOURNAL = {Ann. Inst. Fourier (Grenoble)},
  FJOURNAL = {Universit\'{e} de Grenoble. Annales de l'Institut Fourier},
    VOLUME = {42},
      YEAR = {1992},
    NUMBER = {1-2},
     PAGES = {165--192},
      ISSN = {0373-0956},
   MRCLASS = {57M50 (53C15 57R15 58F05)},
  MRNUMBER = {1162559},
MRREVIEWER = {Nikolai K. Smolentsev},
       DOI = {10.5802/aif.1288},
       URL = {https://doi.org/10.5802/aif.1288},
}

@article {Eliashberg2,
    AUTHOR = {Eliashberg, Yakov},
     TITLE = {Classification of contact structures on {$\bold R^3$}},
   JOURNAL = {Internat. Math. Res. Notices},
  FJOURNAL = {International Mathematics Research Notices},
      YEAR = {1993},
    NUMBER = {3},
     PAGES = {87--91},
      ISSN = {1073-7928},
   MRCLASS = {53C15 (57R15 58F05)},
  MRNUMBER = {1208828},
MRREVIEWER = {Hansj\"{o}rg Geiges},
       DOI = {10.1155/S107379289300008X},
       URL = {https://doi.org/10.1155/S107379289300008X},
}

@article {Barilari,
    AUTHOR = {Agrachev, A. and Barilari, D.},
     TITLE = {Sub-{R}iemannian structures on 3{D} {L}ie groups},
   JOURNAL = {J. Dyn. Control Syst.},
  FJOURNAL = {Journal of Dynamical and Control Systems},
    VOLUME = {18},
      YEAR = {2012},     
    PAGES = {21--44},
      ISSN = {1079-2724},
   MRCLASS = {53C17 (49K21)},
  MRNUMBER = {2902707},
MRREVIEWER = {Thomas Bieske},
       DOI = {10.1007/s10883-012-9133-8},
       URL = {https://doi.org/10.1007/s10883-012-9133-8},
}

@incollection {Massot-LectureNotes,
    AUTHOR = {Massot, Patrick},
     TITLE = {Topological methods in 3-dimensional contact geometry},
 BOOKTITLE = {Contact and symplectic topology},
    SERIES = {Bolyai Soc. Math. Stud.},
    VOLUME = {26},
     PAGES = {27--83},
 PUBLISHER = {J\'{a}nos Bolyai Math. Soc., Budapest},
      YEAR = {2014},
   MRCLASS = {53D10 (53C12 53C15 53C45 58K99)},
  MRNUMBER = {3220940},
MRREVIEWER = {Saibal Ganguli},
       DOI = {10.1007/978-3-319-02036-5\_2},
       URL = {https://doi.org/10.1007/978-3-319-02036-5_2},
}

@article {Massot,
    AUTHOR = {Etnyre, John B. and Komendarczyk, Rafal and Massot, Patrick},
     TITLE = {Tightness in contact metric 3-manifolds},
   JOURNAL = {Invent. Math.},
  FJOURNAL = {Inventiones Mathematicae},
    VOLUME = {188},
      YEAR = {2012},
    NUMBER = {3},
     PAGES = {621--657},
      ISSN = {0020-9910},
   MRCLASS = {53Dxx},
  MRNUMBER = {2917179},
       DOI = {10.1007/s00222-011-0355-2},
       URL = {https://doi.org/10.1007/s00222-011-0355-2},
}

@book {Geiges,
    AUTHOR = {Geiges, Hansj\"{o}rg},
     TITLE = {An introduction to contact topology},
    SERIES = {Cambridge Studies in Advanced Mathematics},
    VOLUME = {109},
 PUBLISHER = {Cambridge University Press, Cambridge},
      YEAR = {2008},
     PAGES = {xvi+440},
      ISBN = {978-0-521-86585-2},
   MRCLASS = {57R17 (53D35)},
  MRNUMBER = {2397738},
MRREVIEWER = {John B. Etnyre},
       DOI = {10.1017/CBO9780511611438},
       URL = {https://doi.org/10.1017/CBO9780511611438},
}

@book {DoCarmo,
    AUTHOR = {do Carmo, Manfredo Perdig\~{a}o},
     TITLE = {Riemannian geometry},
    SERIES = {Mathematics: Theory \& Applications},
      NOTE = {Translated from the second Portuguese edition by Francis
              Flaherty},
 PUBLISHER = {Birkh\"{a}user Boston, Inc., Boston, MA},
      YEAR = {1992},
     PAGES = {xiv+300},
      ISBN = {0-8176-3490-8},
   MRCLASS = {53-01},
  MRNUMBER = {1138207},
MRREVIEWER = {Bang-yen Chen},
       DOI = {10.1007/978-1-4757-2201-7},
       URL = {https://doi.org/10.1007/978-1-4757-2201-7},
}

@article {Eliashberg,
    AUTHOR = {Eliashberg, Y.},
     TITLE = {Classification of overtwisted contact structures on
              {$3$}-manifolds},
   JOURNAL = {Invent. Math.},
  FJOURNAL = {Inventiones Mathematicae},
    VOLUME = {98},
      YEAR = {1989},
    NUMBER = {3},
     PAGES = {623--637},
      ISSN = {0020-9910},
   MRCLASS = {53C15 (53C57 58A15 58C27)},
  MRNUMBER = {1022310},
MRREVIEWER = {Yong-Geun Oh},
       DOI = {10.1007/BF01393840},
       URL = {https://doi.org/10.1007/BF01393840},
}

@article {QDarboux,
    AUTHOR = {Etnyre, John B. and Komendarczyk, Rafal and Massot, Patrick},
     TITLE = {Quantitative {D}arboux theorems in contact geometry},
   JOURNAL = {Trans. Amer. Math. Soc.},
  FJOURNAL = {Transactions of the American Mathematical Society},
    VOLUME = {368},
      YEAR = {2016},
    NUMBER = {11},
     PAGES = {7845--7881},
      ISSN = {0002-9947},
   MRCLASS = {53D35 (53D10 57R17)},
  MRNUMBER = {3546786},
MRREVIEWER = {Stefan Nemirovski},
       DOI = {10.1090/tran/6821},
       URL = {https://doi.org/10.1090/tran/6821},
}

@incollection {bennequin,
    AUTHOR = {Bennequin, Daniel},
     TITLE = {Entrelacements et \'{e}quations de {P}faff},
 BOOKTITLE = {Third {S}chnepfenried geometry conference, {V}ol. 1
              ({S}chnepfenried, 1982)},
    SERIES = {Ast\'{e}risque},
    VOLUME = {107},
     PAGES = {87--161},
 PUBLISHER = {Soc. Math. France, Paris},
      YEAR = {1983},
   MRCLASS = {58F18 (57R15)},
  MRNUMBER = {753131},
MRREVIEWER = {Yakov Eliashberg},
}

@article {giroux1,
    AUTHOR = {Giroux, Emmanuel},
     TITLE = {Convexit\'{e} en topologie de contact},
   JOURNAL = {Comment. Math. Helv.},
  FJOURNAL = {Commentarii Mathematici Helvetici},
    VOLUME = {66},
      YEAR = {1991},
    NUMBER = {4},
     PAGES = {637--677},
      ISSN = {0010-2571},
   MRCLASS = {57R15 (53C15 57R70 58E99 58F09)},
  MRNUMBER = {1129802},
MRREVIEWER = {C. B. Thomas},
       DOI = {10.1007/BF02566670},
       URL = {https://doi.org/10.1007/BF02566670},
}

@article{Geiges1997,
	author = {Hansjörg Geiges},
	title = {Normal contact structures on 3-manifolds},
	journal = {Tohoku Mathematical Journal},
	series = {Second Series},
	volume = {49},
	number = {3},
	pages = {415--422},
	year = {1997},
	doi = {10.2748/tmj/1178225112},
}

@article{Geiges1995,
	author = {Hansjörg Geiges},
	title = {Examples of symplectic 4-manifolds with disconnected boundary of contact type},
	journal = {Bulletin of the London Mathematical Society},
	volume = {27},
	number = {3},
	pages = {278--280},
	year = {1995},
	doi = {10.1112/blms/27.3.278}
}

@incollection {Cortical_BCGPR,
	AUTHOR = {Boscain, Ugo and Chertovskih, Roman and Gauthier, Jean-Paul
	and Prandi, Dario and Remizov, Alexey},
	TITLE = {Cortical-inspired image reconstruction via sub-{R}iemannian
	geometry and hypoelliptic diffusion},
	BOOKTITLE = {S{MAI} 2017---{$8^{\rm e}$} {B}iennale {F}ran\c{c}aise des
	{M}ath\'{e}matiques {A}ppliqu\'{e}es et {I}ndustrielles},
	SERIES = {ESAIM Proc. Surveys},
	VOLUME = {64},
	PAGES = {37--53},
	PUBLISHER = {EDP Sci., Les Ulis},
	YEAR = {2018},
	MRCLASS = {94A08},
	MRNUMBER = {3883978},
	DOI = {10.1051/proc/201864037},
	URL = {https://doi.org/10.1051/proc/201864037},
}

@article {giroux2,
    AUTHOR = {Giroux, Emmanuel},
     TITLE = {Structures de contact en dimension trois et bifurcations des
              feuilletages de surfaces},
   JOURNAL = {Invent. Math.},
  FJOURNAL = {Inventiones Mathematicae},
    VOLUME = {141},
      YEAR = {2000},
    NUMBER = {3},
     PAGES = {615--689},
      ISSN = {0020-9910},
   MRCLASS = {53D35 (37C85 57R17 57R30)},
  MRNUMBER = {1779622},
MRREVIEWER = {Hansj\"{o}rg Geiges},
       DOI = {10.1007/s002220000082},
       URL = {https://doi.org/10.1007/s002220000082},
}

@article {ABBR,
	AUTHOR = {Agrachev, Andrei A. and Baranzini, Stefano and Bellini,
	Eugenio and Rizzi, Luca},
	TITLE = {Quantitative tightness for three-dimensional contact
	manifolds: a sub-{R}iemannian approach},
	JOURNAL = {Nonlinearity},
	FJOURNAL = {Nonlinearity},
	VOLUME = {38},
	YEAR = {2025},
	NUMBER = {11},
	PAGES = {Paper No. 115011, 62},
	ISSN = {0951-7715,1361-6544},
	MRCLASS = {53D10 (53C17 57K33)},
	MRNUMBER = {4989350},
	DOI = {10.1088/1361-6544/ae19be},
	URL = {https://doi.org/10.1088/1361-6544/ae19be},
}
\end{document}